\pdfoutput=1
\RequirePackage{ifpdf}
\ifpdf 
\documentclass[pdftex]{sigma}
\else
\documentclass{sigma}
\fi

\usepackage{leftidx}
\usepackage{tikz-cd}
\usepackage{mathtools}

\newtheorem{tw}{Theorem}[section]

\newtheorem{prop}[tw]{Proposition}

{\theoremstyle{definition}
\newtheorem{defi}[tw]{Definition}
\newtheorem{egg}[tw]{Example}
\newtheorem{rem}[tw]{Remark}}

\numberwithin{equation}{section}

\begin{document}
\allowdisplaybreaks

\newcommand{\arXivNumber}{1710.11440}

\renewcommand{\PaperNumber}{098}

\FirstPageHeading

\ShortArticleName{On the Generalization of Hilbert's Fifth Problem to Transitive Groupoids}

\ArticleName{On the Generalization of Hilbert's Fifth Problem\\ to Transitive Groupoids}

\Author{Pawe{\l} RA\'ZNY}

\AuthorNameForHeading{P.~Ra\'zny}

\Address{Institute of Mathematics, Faculty of Mathematics and Computer Science,\\
Jagiellonian University in Cracow, Poland}
\Email{\href{mailto:pawel.razny@student.uj.edu.pl}{pawel.razny@student.uj.edu.pl}}

\ArticleDates{Received December 02, 2017, in f\/inal form December 21, 2017; Published online December 31, 2017}

\Abstract{In the following paper we investigate the question: when is a transitive topo\-lo\-gical groupoid continuously isomorphic to a Lie groupoid? We present many results on the matter which may be considered generalizations of the Hilbert's f\/ifth problem to this context. Most notably we present a~``solution'' to the problem for proper transitive groupoids and transitive groupoids with compact source f\/ibers.}


\Classification{22A22}

\section{Introduction}
In this short paper we generalize Hilbert's f\/ifth problem (which states that a locally Euclidean topological group is isomorphic as topological groups to a~Lie group and was solved in~\cite{G} and~\cite{MZ}) to the case of transitive groupoids. We prove a couple of versions of Hilbert's f\/ifth problem for transitive groupoids (Theorems~\ref{SE},~\ref{prop} and~\ref{comp}). Many of our main results are written in two versions: one with the weakest assumptions such that the proofs are still valid, and one in which most assumptions are replaced by demanding appropriate spaces to be topological manifolds (which is more in the spirit of Hilbert's f\/ifth problem). We restrict our attention to groupoids with smooth base as otherwise simple counterexamples to the Hilbert's f\/ifth problem for transitive groupoids can be easily found (as described in the f\/inal section of this paper). In light of the examples and results in this paper we feel that topological groupoids with smooth base are natural objects to investigate when it comes to such theorems. Due to the fact that transitive groupoids are an extreme case of groupoids (the opposite extreme case being the totally intransitive groupoids) we feel that the results of this paper combined with a study of totally intransitive groupoids with smooth base can help in solving Hilbert's f\/ifth problem for groupoids.

The paper is split into two parts. The f\/irst part is designed to give all the necessary facts about groupoids and make the paper more self-contained. In the second part we present our main results and give a brief discussion about the assumptions made throughout the paper.

An interested reader can f\/ind a thorough exposition of Hilbert's f\/ifth problem in~\cite{TT}. A~more detailed exposition of groupoid theory and it's vast application (to, e.g., Poisson geometry, symplectic geometry, foliations) can be found in~\cite{McK1,McK2,Moer}. Throughout this paper manifolds (both topological and smooth) are assumed to be Hausdorf\/f and second countable (unless stated otherwise). We do not consider inf\/initely dimensional spaces.

\newpage

\section{Preliminaries}
\subsection{Some topology}
In order to make this short paper as self-contained as possible, we start by recalling some basic topological notions and properties which are used in subsequent sections. We begin with some well known properties of quotient maps (identif\/ications). Recall that a continuous surjective map $p\colon X\rightarrow Y$ is a \emph{quotient map} if a subset $U$ of $Y$ is open in $Y$ if and only if $p^{-1}(U)$ is open in~$X$. Equivalently, $Y$ is the quotient space of $X$ with respect to the relation $\sim$ given by $x\sim y$ if and only if there exists a point $z\in Y$ such that $x,y\in p^{-1}(z)$.
\begin{prop} A surjective continuous map which is either closed or open is a quotient map.
\end{prop}

\begin{defi} Given a quotient map $p\colon X\rightarrow Y$ a subset $U$ of $X$ is called saturated if $x\in p(U)$ implies $p^{-1}(x)\subset U$.
\end{defi}

\begin{prop} Given a quotient map $p\colon X\rightarrow Y$ the restriction $p|_U\colon U\rightarrow p(U)$ to a closed or open saturated subset $U$ of $X$ is also a quotient map.
\end{prop}

\begin{prop}\label{PtQ} Given a quotient map $p\colon X\rightarrow Y$ and another continuous map $f\colon X\rightarrow Z$ which is constant on the fibers of $p$ there is a unique continuous map $\bar{f}\colon Y\rightarrow Z$ such that the following diagram commutes:
\[
\begin{tikzcd}
X \arrow[rd, "p"] \arrow[r, "f"] & Z \\
& Y.\arrow[u, "\bar{f}"]
\end{tikzcd}
\]
\end{prop}

All of the above results along with a good revision of quotient maps can be found in \cite{L}. We also wish to recall the following two versions of the closed mapping theorem:
\begin{tw}[{\cite[Lemma 4.50]{L}}] A continuous map $f\colon X\rightarrow Y$, where $X$ is compact and $Y$ is Hausdorff, is a closed map.
\end{tw}

\begin{tw}[{\cite[Theorem 4.95]{L}}] A proper continuous map $f\colon X\rightarrow Y$, where $X$ and $Y$ are locally compact and Hausdorff, is a closed map.
\end{tw}

We now recall the notion of sequence covering which will be used extensively throughout this paper. More information on this subject can be found in~\cite{Si}.

\begin{defi} A continuous surjective map $f\colon X\rightarrow Y$ is said to have the sequence covering property if for each sequence $\{y_n\}_{n\in\mathbb{N}}$ convergent to $y$ in $Y$ there exists a sequence $\{x_n\}_{n\in\mathbb{N}}$ convergent to $x$ in $X$ satisfying $x_n\in f^{-1}(y_n)$ and $x\in f^{-1}(y)$.
\end{defi}

\begin{prop}\label{SCP}
Any continuous open surjective map has the sequence covering property.
\end{prop}

Finally, we recall the following well known fact:
\begin{prop}[{\cite[Theorem 1.6.14 and Proposition 1.6.15]{E}}]\label{seq} Let $f\colon X\rightarrow Y$ be a function from a first countable topological space $X$ to a topological space~$Y$. Then~$f$ is continuous if and only if it is sequentially continuous.
\end{prop}

\subsection{Groupoids}
We give a brief recollection of some basic notions concerning groupoids. Let us start by giving the def\/inition:
\begin{defi} A groupoid $\mathcal{G}$ is a small category in which all the morphisms are isomorphisms. Let us denote by $\mathcal{G}_0$ the set of objects of this category (also called the base of $\mathcal{G}$) and by $\mathcal{G}_1$ the set of morphisms of this category. This implies the existence of the following f\/ive structure maps:
\begin{enumerate}\itemsep=0pt
\item[1)] the source map $s\colon \mathcal{G}_1\rightarrow\mathcal{G}_0$ which associates to each morphism its source,
\item[2)] the target map $t\colon \mathcal{G}_1\rightarrow\mathcal{G}_0$ which associates to each morphism its target,
\item[3)] the identity map ${\rm Id}\colon \mathcal{G}_0\rightarrow\mathcal{G}_1$ which associates to each object the identity over that object,
\item[4)] the inverse map $i\colon \mathcal{G}_1\rightarrow\mathcal{G}_1$ which associates to each morphism its inverse,
\item[5)] the multiplication (composition) map $\circ\colon \mathcal{G}_2\rightarrow\mathcal{G}_1$ which associates to each composable pair of morphisms its composition ($\mathcal{G}_2$ is the set of composable pairs).
\end{enumerate}
 A groupoid endowed with topologies on $\mathcal{G}_1$ and $\mathcal{G}_0$ which make all the structure maps continuous is called a topological groupoid. If additionally $\mathcal{G}_0$ and $\mathcal{G}_1$ are smooth manifolds, the source map is a surjective submersion and all the structure maps are smooth then the topological groupoid is called a Lie groupoid.
\end{defi}

\begin{rem} Note that the identity map and the target map restricted to the image of the identity map are inverses and so the identity map is an embedding. Hence, we can identify $\mathcal{G}_0$ with the image of the identity map. This in particular means that we can treat~$\mathcal{G}_0$ as a subspace of~$\mathcal{G}_1$.
\end{rem}

We denote by $\mathcal{G}_x$ the f\/ibers of the source map (source f\/ibers) and by $\mathcal{G}^x$ the f\/ibers of the target map (target f\/ibers). We also denote by $\mathcal{G}^y_x$ the set of morphisms with source $x$ and target~$y$. For a Lie groupoid $\mathcal{G}_x$, $\mathcal{G}^x$ and $\mathcal{G}^y_x$ are all closed embedded submanifolds of $\mathcal{G}_1$. What is more $\mathcal{G}^x_x$ are Lie groups. If~$V$ and~$U$ are subsets of $\mathcal{G}_0$ then we denote by $\mathcal{G}_U^V$ the set of all morphisms with source in $U$ and target in $V$. We present a couple of examples which will be of further interest to us:
\begin{egg} Given a smooth manifold (resp.\ topological space) $M$ there is a structure of a~Lie groupoid (resp.\ topological groupoid) with base~$M$ on $M\times M$. For this structure the source and target maps are projections on to the f\/irst and second factor. We call the groupoid $M\times M$ over M the pair groupoid.
\end{egg}

\begin{egg}
Let $G$ be a topological (resp.\ Lie) group and let $M$ be a topological space (resp.\ smooth manifold). $M\times G\times M$ is a topological (resp.\ Lie) groupoid over $M$ with source and target maps given by projections onto the f\/irst and third factor and composition law given by the formula:
\begin{gather*}
(x,g,y)\circ (y,h,z)=(x,gh,z).
\end{gather*}
Groupoids of this form are called trivial groupoids.
\end{egg}

There are many more examples of groupoids most of which are much more complicated then the ones just shown. The above examples will be used later on as they already exhibit some of the problems one faces when dealing with groupoids that don't arise for groups. In what follows we are also going to need a notion of morphism of groupoids:

\begin{defi} A continuous (resp.\ smooth) morphism of topological (resp.\ Lie) groupoids is a pair $(F,f)\colon \mathcal{G}\rightarrow\mathcal{H}$ where $F\colon \mathcal{G}_1\rightarrow\mathcal{H}_1$ and $f\colon \mathcal{G}_0\rightarrow\mathcal{H}_0$ are continuous (resp.\ smooth) maps which commute with the structure maps. If in addition both $F$ and $f$ are homeomorphisms (resp.\ dif\/feomorphisms) then $(F,f)$ is a continuous (resp.\ smooth) isomorphism. Furthermore, if f is the identity on $\mathcal{G}_0$ then the morphism $(F,f)$ is called a base preserving morphism.
\end{defi}

Throughout this paper we are going to use several special classes of groupoids:
\begin{defi} A groupoid $\mathcal{G}$ is said to be transitive if for each pair of points $x,y\in\mathcal{G}_0$ there exists a morphism with source $x$ and target~$y$.
\end{defi}

\begin{defi} A topological groupoid is said to be proper if the map $(s,t)\colon \mathcal{G}_1\rightarrow\mathcal{G}_0\times\mathcal{G}_0$ is proper.
\end{defi}

\begin{defi}
A topological groupoid is principal if:
\begin{enumerate}\itemsep=0pt
\item[1)] the restriction of the target map to any source f\/iber is a quotient map (we write $t_x$ for the restriction of the target map to the source f\/iber over $x\in\mathcal{G}_0$),
\item[2)] for each $x\in\mathcal{G}_0$ the division map $\delta_x\colon \mathcal{G}_x\times\mathcal{G}_x\rightarrow\mathcal{G}_1$ def\/ined by the formula $\delta_x(g,h)=g\circ h^{-1}$ is a quotient map.
\end{enumerate}

\begin{rem} It is worth noting that for a transitive topological groupoid $\mathcal{G}$ and a given source f\/iber $\mathcal{G}_x$ composition with $h\colon x\rightarrow y$ gives a homeomorphism $\tilde{h}\colon \mathcal{G}_x\rightarrow \mathcal{G}_y$ since it has an inverse (composition with $h^{-1}$) and conjugation by $h$ ($hgh^{-1}$ for $g\in\mathcal{G}_x^x$) gives a continuous isomorphism of topological groups $h^*\colon \mathcal{G}^x_x\rightarrow \mathcal{G}^y_y$ since it has an inverse (conjugation by $h^{-1}$). Hence, we write ``$\mathcal{G}_x$ (resp.~$\mathcal{G}_x^x$) has property $P$'' as shorthand for~``$\mathcal{G}_x$ (resp.\ $\mathcal{G}^x_x$) has property $P$ for some $x\in\mathcal{G}_0$ and equivalently for any $x\in\mathcal{G}_0$'' whenever this remark can be applied. Moreover, since $t_y(\tilde{h})=t_x$ and $\delta_y(\tilde{h},\tilde{h})=\delta_x$ we write ``$t_x$ (resp.\ $\delta_x$) has property $P$'' as shorthand for ``$t_x$ (resp.~$\delta_x$) has property $P$ for some $x\in\mathcal{G}_0$ and equivalently for any $x\in\mathcal{G}_0$'' whenever this remark can be applied.
\end{rem}

\end{defi}
\begin{defi} A topological groupoid is locally trivial if it is transitive and each point $x\in\mathcal{G}_1$ has an open neighbourhood $U$ such that $\mathcal{G}^U_U$ is base preserving continuously isomorphic to the trivial groupoid $U\times\mathcal{G}_x^x\times U$.
\end{defi}

\subsection{Cartan principal bundles}
In this section we give a brief recollection of principal bundles, Cartan principal bundles, how they relate to one another and their connection to transitive topological groupoids. A more detailed exposition of this subject can be found in \cite{McK1,McK2,P}.
\begin{defi} A Cartan principal bundle is a quadruple $(P,B,G,\pi)$, where $P$ and $B$ are topological spaces, $G$ is a topological group acting freely on $P$ and $\pi\colon P\rightarrow B$ is a surjective continuous map, with the following properties:
\begin{enumerate}\itemsep=0pt
\item[1)] $\pi$ is a quotient map with f\/ibers coinciding with the orbits of the action of $G$ on $P$,
\item[2)] the division map $\delta\colon P_{\pi}\rightarrow G$ with domain $P_{\pi}:=\{(u,v)\in P\times P\,|\, \pi(u)=\pi(v)\}$ def\/ined by the property $\delta(ug,u)=g$ is continuous.
\end{enumerate}
\end{defi}

We are also going to need a notion of morphism between such bundles:
\begin{defi}

A morphism of Cartan principal bundles is a triple $(F,f,\phi)\colon (P,B,G,\pi)\rightarrow (P',B',G',\pi')$, where $F\colon P\rightarrow P'$ and $f\colon B\rightarrow B'$ are continuous functions and $\phi\colon G\rightarrow G'$ is a~continuous morphism of topological groups such that:
\begin{gather*}
\pi'\circ f=f\circ\pi, \qquad F(pg)=F(p)\phi(g)
\end{gather*}
for $p\in P$ and $g\in G$. A morphism of Cartan principal bundles is said to be base preserving if $B=B'$ and $f={\rm Id}_B$.
\end{defi}

A stronger and better known object is the following:
\begin{defi} A principal bundle is a quadruple $(P,B,G,\pi)$, where $P$ and $B$ are topological spaces, $G$ is a topological group acting freely on $P$ and $\pi\colon P\rightarrow B$ is a surjective continuous map, with the following properties:
\begin{enumerate}\itemsep=0pt
\item[1)] the f\/ibers of $\pi$ coincide with the orbits of the action of $G$,
\item[2)] (local triviality) There is an open covering $U_i$ of $B$ and continuous maps $\sigma_i\colon U_i\rightarrow P$ such that $\pi\circ\sigma_i={\rm Id}_{U_i}$.
\end{enumerate}
A principal bundle is said to be smooth if $P$ and $B$ are smooth manifolds, $G$ is a Lie group, and the action, projection and $\sigma_i$ are smooth maps.
\end{defi}

We sometimes refer to principal bundles as continuous principal bundles. It is known that a~principal bundle is a Cartan principal bundle and that a Cartan principal bundle which is locally trivial is principal (cf.~\cite{McK1} and \cite{P}). We also present the following important result from~\cite{P}:
\begin{tw}\label{Slice} A Cartan principal bundle $(P,B,G,\pi)$ with $G$ a Lie group and $P$ Tychonoff is locally trivial.
\end{tw}

We also wish to recall the following well known equivalence (cf.~\cite{MW}):
\begin{tw} A principal bundle $(P,B,G,\pi)$ with $G$ a Lie group and~$B$ a smooth manifold is continuously isomorphic through a~$($base preserving isomorphism$)$ to a unique $($up to smooth isomorphism of smooth principal bundles$)$ smooth principal bundle.
\end{tw}

We are now going to present important constructions from \cite{McK1} and \cite{McK2} which relate the above notions of principal and Cartan principal bundles to locally trivial and principal groupoids respectively. Given a principal groupoid $\mathcal{G}$ the quadruple $(\mathcal{G}_x,\mathcal{G}_0,\mathcal{G}^x_x,t_x)$ constitutes a Cartan principal bundle for any point $x\in\mathcal{G}_0$ (this is called the \emph{vertex bundle of} $\mathcal{G}$ \emph{at} $x$). It is easy to see that given a morphism of groupoids $(F,f)\colon \mathcal{G}\rightarrow\mathcal{G}'$ the restriction of the map~$F$ to~$\mathcal{G}_x$ gives a morphism of bundles $F|_{\mathcal{G}_x}\colon \mathcal{G}_x\rightarrow\mathcal{G}_{f(x)}$. It is also worth noting that even though this construction is dependent on the choice of~$x$ all the vertex bundles are continuously isomorphic by use of translations (cf.~\cite{McK1}). On the other hand given a Cartan principal bundle $(P,B,G,\pi)$ there exists a structure of a topological groupoid over $B$ on $(P\times P)\slash G$ (this is called the \emph{gauge groupoid of} $(P,B,G,\pi)$). Furthermore, a morphism of principal bundles $(F,f,\phi)\colon (P,B,G,\pi)\rightarrow (P',B',G',\pi')$ induces a morphism of gauge groupoids $F^*$ def\/ined by $F^*([(u,v)])=[F(u),F(v)]$. It is apparent from the form of the induced morphisms that a base preserving morphism of Cartan principal bundles induces a base preserving morphism of the corresponding gauge groupoids and that a base preserving morphism of principal groupoids induces a base preserving morphism of vertex bundles. We give the following 3 theorems which were proven in \cite{McK1} and \cite{McK2}:
\begin{tw}\label{CPB} The constructions above are mutually inverse $($up to a continuous base preserving isomorphism$)$ and give a one to one correspondence between continuous isomorphism classes of Cartan principal bundles and continuous isomorphism classes of principal groupoids.
\end{tw}

\begin{tw}\label{PB} The constructions above give a one to one correspondence between continuous isomorphism classes of principal bundles and continuous isomorphism classes of locally trivial groupoids.
\end{tw}
\begin{tw}\label{SPB} The constructions above give a one to one correspondence between smooth isomorphism classes of smooth principal bundles and smooth isomorphism classes of locally trivial Lie groupoids.
\end{tw}

\begin{rem} In the previous theorem one can weaken local triviality to transitivity since in the case of Lie groupoids these notions are equivalent (cf.~\cite[Corollary~1.9]{McK1}).
\end{rem}

\section{Hilbert's f\/ifth problem for transitive groupoids}
\subsection{Main results}
In the following section by ``unique Lie groupoid'' we mean unique up to a smooth base preserving isomorphism. We start with the following observation.
\begin{tw}\label{2TH5}
Let $\mathcal{G}$ be a principal groupoid with a smooth base $\mathcal{G}_0$ and Tychonoff source fi\-bers~$\mathcal{G}_x$ for which the topological groups $\mathcal{G}_x^x$ are locally Euclidean. Then $\mathcal{G}$ is continuously isomorphic to a unique Lie groupoid through a base preserving isomorphism.
\end{tw}

\begin{proof} By the correspondence in Theorem \ref{CPB} we can associate to $\mathcal{G}$ a Cartan principal bundle. Due to Hilbert's f\/ifth problem $\mathcal{G}^x_x$ is a Lie group. Since the source f\/ibers are Tychonof\/f we can see by Theorem~\ref{Slice} that this bundle is locally trivial and hence its gauge groupoid (which is continuously isomorphic to $\mathcal{G}$ through a base preserving isomorphism) is locally trivial as well thanks to the correspondence in Theorem~\ref{PB}. Hence, $\mathcal{G}$ is locally trivial. Furthermore, since this bundle is a continuous principal bundle of Lie groups over a smooth manifold it is continuously isomorphic to a unique smooth principal bundle. This continuous isomorphism induces a continuous isomorphism between the gauge groupoids of these two bundles. It now suf\/f\/ices to note that by Theorem~\ref{SPB} the gauge groupoid of a smooth principal bundle is a Lie groupoid.
\end{proof}

Note that if $\mathcal{G}_x$ is assumed to be a topological manifold then it is in fact Tychonof\/f and hence the above theorem can be applied. In order to further generalize this theorem we present an important technical result:
\begin{prop}\label{eq} Let $\mathcal{G}$ be a transitive topological groupoid with $\mathcal{G}_1$ first countable. Then the following conditions are equivalent:
\begin{enumerate}\itemsep=0pt
\item[$1)$] $t_x$ is a quotient map,
\item[$2)$] $t_x$ is open,
\item[$3)$] $t_x$ has the sequence covering property,
\item[$4)$] $\delta_x$ is a quotient map,
\item[$5)$] $\delta_x$ is open.
\end{enumerate}
Furthermore, if $(s,t)\colon \mathcal{G}_1\rightarrow \mathcal{G}_0\times\mathcal{G}_0$ is a quotient map, then the above properties hold.
\end{prop}

\begin{proof} The equivalence of conditions (1) and (2) as well as the equivalence of conditions (4) and (5) can be found in~\cite{McK1}. That~(2) implies~(3) is immediate from Proposition \ref{SCP}.

For (3) implies (1) let us denote by $t'_x\colon \mathcal{G}_x\rightarrow\mathcal{G}_x\slash\mathcal{G}^x_x$ the quotient map of the group action of~$\mathcal{G}_x^x$ on~$\mathcal{G}_x$. Using Proposition \ref{PtQ} we get the following commutative diagram:
\begin{equation}\label{diag1}
\begin{tikzcd}
\mathcal{G}_x \arrow[rd, "t'_x"] \arrow[r, "t_x"] & \mathcal{G}_0 \\
& \mathcal{G}_x\slash\mathcal{G}^x_x,\arrow[u, "f"]
\end{tikzcd}
\end{equation}
where $f$ is a continuous bijection. By Proposition~\ref{seq} it is suf\/f\/icient to prove that $f^{-1}$ is sequentially continuous (since $\mathcal{G}_0$ is f\/irst countable as a subspace of a f\/irst countable space $\mathcal{G}_1$). Let us take a sequence $\{y_n\}_{n\in\mathbb{N}}$ convergent to $y$ in $\mathcal{G}_0$. We prove that its image $f^{-1}(y_n)$ is convergent to $f^{-1}(y)$. Using the sequence covering property of~$t_x$ we get a sequence \smash{$\{g_n\colon x\rightarrow y_n\}_{n\in\mathbb{N}}$} convergent to $g\colon x\rightarrow y$. Due to the continuity of~$t'_x$ the sequence $t'_x(h_n)$ converges to~$t'_x(h)$. This sequence and its limit are by the commutativity of the diagram and bijectivity of~$f$ the image of~$\{y_n\}_{n\in\mathbb{N}}$ and $y$ respectively through $f^{-1}$. Hence, $f^{-1}$ is continuous which in turn implies that~$f$ is a homeomorphism and so $t_x$ is a quotient map.

The proof of (3) implies (4) is similar to the previous implication. Let us f\/irst note that it suf\/f\/ices to prove that the multiplication map restricted to $\mathcal{G}_x\times\mathcal{G}^x$ is a quotient map (since $\delta_x$ is a composition of the multiplication map and the inverse map which is a~homeomorphism). Let us also note that~$\mathcal{G}_x^x$ acts on~$\mathcal{G}_x\times\mathcal{G}^x$ by
\begin{gather*}
(h_1,h_2)g=\big(h_1g,g^{-1}h_2\big).
\end{gather*}
We prove that the orbits of this action are precisely the f\/ibers of the multiplication map restricted to $\mathcal{G}_x\times\mathcal{G}^x$. It is apparent that an orbit of this action is contained in a f\/iber of the restricted multiplication map (since the composition $h_1gg^{-1}h_2$ is equal to $h_1h_2$). On the other hand given two elements $(f,g)$ and $(f',g')$ in a single f\/iber of the multiplication map we have $fg=f'g'$. This implies that ${\rm Id}_x=f^{-1}f'g'g^{-1}$ and hence $g'g^{-1}$ is inverse to $f^{-1}f'$. This means that by acting on $(f,g)$ with $f^{-1}f'$ we get $(f',g')$ which in turn implies that $(f,g)$ and $(f',g')$ belong to the same orbit of the group action. Hence, each f\/iber of the restricted multiplication is contained in some orbit of the group action (and so the f\/ibers and orbits coincide). Let us denote by $\circ_x$ this restricted multiplication map and by $\circ'_x\colon \mathcal{G}_x\times\mathcal{G}^x\rightarrow (\mathcal{G}_x\times\mathcal{G}^x)\slash\mathcal{G}^x_x$ the quotient map of the group action. As in the previous case we have the following commutative diagram:
\begin{equation}
\begin{tikzcd}
\mathcal{G}_x\times\mathcal{G}^x \arrow[rd, "\circ'_x"] \arrow[r, "\circ_x"] & \mathcal{G}_1 \\
& (\mathcal{G}_x\times\mathcal{G}^x)\slash\mathcal{G}^x_x.\arrow[u, "f"]
\end{tikzcd}
\end{equation}
The map $f$ is again bijective and continuous and we prove the continuity of $f^{-1}$. Since $\mathcal{G}_1$ is f\/irst countable it is suf\/f\/icient to prove sequential continuity by Proposition \ref{seq}. Let us take a~sequence $\{g_n\colon y'_n\rightarrow y_n\}_{n\in\mathbb{N}}$ convergent in $\mathcal{G}_1$ to $g\colon y'\rightarrow y$. This implies in particualr that the sequence $\{y_n\}_{n\in\mathbb{N}}$ converges to~$y$ in~$\mathcal{G}_0$. Hence, we can use the sequence covering property of~$t_x$ to produce a sequence $\{h_n\colon x\rightarrow y_n\}_{n\in\mathbb{N}}$ convergent to $h\colon x\rightarrow y$ in $\mathcal{G}_x$. This allows us to def\/ine a sequence $(h_n,h_n^{-1}g_n)$ which by continuity of multiplication is convergent to~$(h,h^{-1}g)$ in~$\mathcal{G}_x\times\mathcal{G}^x$. As before, using the commutativity of the diagram, bijectivity of $f$ and continuity of~$\circ'_x$ we conclude that the classes of this sequence represent the image of the initial sequence~$g_n$ and converge to~$g$ in~$(\mathcal{G}_x\times\mathcal{G}^x)\slash\mathcal{G}^x_x$. This implies that~$f$ is a homeomorphism and that~$\circ_x$ is in fact a quotient map.

For (5) implies (3) let us take a sequence $\{y_n\}_{n\in\mathbb{N}}$ convergent to~$y$. Since $\delta_x$ is open, so is the multiplication $\circ_x$. Hence, $\circ_x$ has the sequence covering property. Let $\{(g_n,h_n)\}_{n\in\mathbb{N}}$ be the sequence covering ${\rm Id}_{y_n}$ (which by continuity of the identity structure map converges to~${\rm Id}_y$) through~$\circ_x$. This implies that $\{g_n\}_{n\in\mathbb{N}}$ is a sequence covering $\{y_n\}_{n\in\mathbb{N}}$ through~$t_x$.

Finally, if $(s,t)$ is a quotient map then $(s,t)|_{\mathcal{G}_x}=(x,t_x)\colon \mathcal{G}_x\rightarrow\{x\}\times\mathcal{G}_0$ is also a quotient map (since $\mathcal{G}_x$ is a saturated closed subset of $\mathcal{G}_1$). This implies that $t_x$ is a quotient map.
\end{proof}

\begin{rem} This in particular means that a transitive groupoid $\mathcal{G}$ with $\mathcal{G}_1$ f\/irst countable is a principal groupoid if and only if $t_x$ is a quotient map.
\end{rem}
\begin{rem} One could add a sixth condition: ``$\delta_x$ has the sequence covering property'' to the previous theorem and prove its equivalence in a similar fashion. However, this condition is of little importance to this paper and hence we skip the proof for the sake of brevity.
\end{rem}
Using the two previous results we immediately arrive at the following conclusion:
\begin{tw}\label{strong}
Let $\mathcal{G}$ be a transitive topological groupoid with a smooth base $\mathcal{G}_0$, Tychonoff source fibers $\mathcal{G}_x$ and first countable space of morphisms $\mathcal{G}_1$ for which the topological groups $\mathcal{G}_x^x$ are locally Euclidean and the map~$t_x$ is a quotient map $($or equivalently has the sequence covering property$)$. Then $\mathcal{G}$ is continuously isomorphic to a unique Lie groupoid through a base preserving isomorphism.
\end{tw}

It is interesting that this result makes only the f\/irst countability assumptions on the topology of $\mathcal{G}_1$. We proceed to give a number of corollaries of this seemingly simple result:
\begin{tw}\label{SE} Let $\mathcal{G}$ be a transitive topological groupoid with a smooth base $\mathcal{G}_0$ for which the spaces~$\mathcal{G}_x^x$, $\mathcal{G}_1$ and~$\mathcal{G}_x$ are topological manifolds and the map $t_x$ is a quotient map $($or equivalently has the sequence covering property$)$. Then $\mathcal{G}$ is continuously isomorphic to a unique Lie groupoid through a base preserving isomorphism.
\end{tw}

\begin{proof} Apply the previous theorem.
\end{proof}

We present the following two forms of the Hilbert's f\/ifth problem for proper transitive groupoids (a stronger one and an elegant one):
\begin{tw}
Let $\mathcal{G}$ be a proper transitive topological groupoid with a smooth base $\mathcal{G}_0$, Tychonoff locally compact source fibers $\mathcal{G}_x$, a locally compact Hausdorff first countable space of morphisms~$\mathcal{G}_1$ for which the topological groups~$\mathcal{G}_x^x$ are locally Euclidean. Then $\mathcal{G}$ is continuously isomorphic to a unique Lie groupoid through a base preserving isomorphism.
\end{tw}

\begin{proof} Under these assumptions using the closed mapping theorem to $(s,t)$ we conclude that it is in fact closed and hence a quotient map. This implies that $t_x$ is a quotient map as (by Proposition~\ref{eq}) and hence we can apply Theorem~\ref{strong}.
\end{proof}

\begin{tw}\label{prop} Let $\mathcal{G}$ be a proper transitive topological groupoid with a smooth base~$\mathcal{G}_0$, for which the spaces~$\mathcal{G}_x^x$,~$\mathcal{G}_x$ and~$\mathcal{G}_1$ are topological manifolds. Then~$\mathcal{G}$ is continuously isomorphic to a unique Lie groupoid through a base preserving isomorphism.
\end{tw}

\begin{proof} Apply the previous theorem.
\end{proof}

\begin{rem} This in particular means (when combined with the results of~\cite{AG}) that any transitive topological groupoid which satisf\/ies the assumptions of Theorem \ref{prop} is continuously isomorphic to a real analytic transitive groupoid.
\end{rem}
Finally, Theorem \ref{strong} also solves the problem for groupoids with compact source f\/ibers (and in particular for groupoids with $\mathcal{G}_1$ compact):
\begin{tw}
Let $\mathcal{G}$ be a transitive topological groupoid with a smooth base~$\mathcal{G}_0$, Tychonoff compact source fibers~$\mathcal{G}_x$ for which the topological groups $\mathcal{G}_x^x$ are locally Euclidean and~$\mathcal{G}_1$ is first countable. Then~$\mathcal{G}$ is continuously isomorphic to a unique Lie groupoid through a base preserving isomorphism.
\end{tw}

\begin{proof} By Theorem \ref{strong} it is suf\/f\/icient to prove that $t_x$ is a quotient map. Using the closed mapping theorem we conclude that~$t_x$ is closed and hence a quotient map.
\end{proof}

\begin{tw}\label{comp} Let $\mathcal{G}$ be a transitive topological groupoid with a smooth base~$\mathcal{G}_0$, for which the spaces $\mathcal{G}_x^x$ and $\mathcal{G}_1$ are topological manifolds and~$\mathcal{G}_x$ is a compact topological manifold. Then~$\mathcal{G}$ is continuously isomorphic to a unique Lie groupoid through a base preserving isomorphism.
\end{tw}

\begin{proof}
Apply the previous theorem.
\end{proof}

\subsection{Some remarks considering assumptions}
We start by showing why it is necessary for $\mathcal{G}_0$ to be smooth. Let $\mathcal{G}_0$ be a~topological manifold which does not admit any smooth structure (e.g., the celebrated $E_8$ 4-manifold). We then take~$\mathcal{G}$ to be the pair groupoid over~$\mathcal{G}_0$. It is apparent that even though in this example~$\mathcal{G}_x^x$, $\mathcal{G}_1$, $\mathcal{G}_x$ and~$\mathcal{G}^x$ are Hausdorf\/f second-countable topological manifolds and the map $t_x$ is a quotient map it is not continuously isomorphic to a Lie groupoid (since this would imply that $\mathcal{G}_0$ is homeomorphic to some smooth manifold). The above groupoid is a counterexample to all the theorems in the previous section if the smoothness assumption on the base is omitted.
\begin{rem} We also wish to note that this assumption is natural in the following sense. One can consider a Lie group $G$ as a Lie groupoid with $\mathcal{G}_0=\{x\}$ and $\mathcal{G}_1=\mathcal{G}^x_x=G$. So if we state the Hilbert's f\/ifth problem in the language of groupoids our assumptions would be $\mathcal{G}_1$ is a~topological manifold and $\mathcal{G}_0$ is a point (and hence a smooth manifold).
\end{rem}

 Assumption that $\mathcal{G}^x_x$ is a topological manifold might not be necessary in Theorems~\ref{SE},~\ref{prop} and~\ref{comp}. However, for the time being it seems hard to get rid of it. We note that this assumption is superf\/luous assuming the validity of Hilbert--Smith conjecture (cf.~\cite{TT}).

In the previous section we explored an approach to Lie groupoids presented in~\cite{McK2}. We wish to address a dif\/ferent approach which is found in~\cite{Moer}:
\begin{defi}[Moerdijk \cite{Moer}] A topological groupoid $\mathcal{G}$ which satisf\/ies the following conditions:
\begin{enumerate}\itemsep=0pt
\item[1)] $\mathcal{G}_0$ is a smooth manifold,
\item[2)] $\mathcal{G}_1$ is a possibly non-Hausdorf\/f and not second countable smooth manifold,
\item[3)] the structure maps are smooth,
\item[4)] the source map is a surjective submersion with Hausdorf\/f f\/ibers,
\end{enumerate}
is called a Lie groupoid.
\end{defi}This def\/inition is somewhat weaker than the one previously considered. Theorems~\ref{strong},~\ref{SE} and~\ref{comp} f\/it nicely with this def\/inition. However, there is a slight problem with Theorem \ref{prop} as it uses the assumption which states that $\mathcal{G}_1$ is Hausdorf\/f which is somewhat restricting when considering the former def\/inition. However, in this context one usually uses a stronger def\/inition of transitivity:
\begin{defi}(Moerdijk~\cite{Moer}) A Lie groupoid is called transitive if $(s,t)\colon \mathcal{G}_1\rightarrow\mathcal{G}_0\times\mathcal{G}_0$ is a~surjective submersion.
\end{defi}
Taking this into consideration we arrive at a version of our main Theorem which is more f\/itting in this case:
\begin{tw}
Let $\mathcal{G}$ be a transitive topological groupoid with a smooth base $\mathcal{G}_0$, Tychonoff source fibers $\mathcal{G}_x$ for which the topological groups $\mathcal{G}_x^x$ are locally Euclidean, the space~$\mathcal{G}_1$ is first countable and the map $(s,t)\colon \mathcal{G}_1\rightarrow\mathcal{G}_0\times\mathcal{G}_0$ is a quotient map. Then $\mathcal{G}$ is continuously isomorphic to a unique Lie groupoid through a base preserving isomorphism.
\end{tw}

\begin{proof}
This follows from Proposition \ref{eq} and Theorem \ref{strong}.
\end{proof}

\begin{tw} Let $\mathcal{G}$ be a transitive topological groupoid with a smooth base $\mathcal{G}_0$, for which the group $\mathcal{G}_x^x$ is a topological manifold, $\mathcal{G}_x$ is a $($possibly not second countable$)$ topological manifold, $\mathcal{G}_1$~is a $($possibly not second countable and non-Hausdorff$)$ topological manifold and the map $(s,t)\colon \mathcal{G}_1\rightarrow\mathcal{G}_0\times\mathcal{G}_0$ is a quotient map. Then~$\mathcal{G}$ is continuously isomorphic to a unique Lie groupoid through a base preserving isomorphism.
\end{tw}
\begin{proof} Apply the previous theorem.
\end{proof}

\begin{rem} Note that in the above theorems the resulting groupoids are Lie with respect to both def\/initions.
\end{rem}

\pdfbookmark[1]{References}{ref}
\LastPageEnding

\end{document}